%% file: main.tex
\documentclass[10pt, article]{amsart}
\usepackage[margin=1.2in]{geometry} 
\pdfoutput=1
\usepackage{ae} 
\usepackage[T1]{fontenc}
\usepackage[cp1250]{inputenc}
\usepackage{amsmath}
\usepackage{amssymb, amsfonts,amscd,verbatim}
\usepackage{mathtools}
\usepackage{MnSymbol}
\usepackage{tikz}
\usetikzlibrary{positioning}
\usepackage{graphicx}
\usepackage{tikz-cd}
\usepackage{comment}

\usepackage[pagewise]{lineno} 

\usepackage[normalem]{ulem}
\usepackage{indentfirst}
\usepackage{latexsym}
\input xy
\xyoption{all}

\usepackage[bookmarks=true,
              bookmarksnumbered=true, breaklinks=true,
              pdfstartview=FitH, hyperfigures=false,
              plainpages=false, naturalnames=true,
              colorlinks=true,pagebackref=false,
              pdfpagelabels]{hyperref}

\usepackage{amsmath}    
\newcommand{\ignore}[1]{ }

\newcommand{\Dn}{\mathcal{D}_n} %
\newcommand{\D}[1]{\mathcal{D}_{{#1}}} %
\newcommand{\db}{d_{\mathcal{B}}}

\newcommand{\dis}{\textrm{dis}}

\theoremstyle{plain}
\newtheorem{Pocz}{Poczatek}[section]
\newtheorem{Proposition}[Pocz]{Proposition}

\newtheorem{Theorem}[Pocz]{Theorem}

\newtheorem{Lemma}[Pocz]{Lemma}

\newtheorem{Example}[Pocz]{Example}

\theoremstyle{definition}
\newtheorem{Definition}[Pocz]{Definition}
\newtheorem{Remark}[Pocz]{Remark}

\newtheorem*{theorem*}{Theorem}

\def\RR{{\mathbb R}}

\def\GG{{\mathbb G}}

\def\NN{{\mathbb N}}

\def\f{{\varphi}}
\def\l{\lambda}

\def\UU{{\mathcal{U}}}
\def\VV{{\mathcal{V}}}

\numberwithin{equation}{section}

\author{Atish ~ Mitra}
\address{Montana Tech, USA}
\email{amitra@mtech.edu}

\author{\v Ziga Virk}
\address{University of Ljubljana, Slovenia, and Institute IMFM, Ljubljana, Slovenia}
\email{ziga.virk@fri.uni-lj.si}

\title[Geometric embeddings of spaces of persistence diagrams with explicit distortions]%
  {Geometric embeddings of spaces of persistence diagrams with explicit distortions}
\thanks{The authors have been supported by Slovenian Research Agency grant No. BI-US/18-20-060. The second named author was also supported by Slovenian Research Agency grants J1-4001, J1-4031, and P1-0292.}

\date{ \today
}


\includecomment{comment}
\begin{document}


\begin{abstract}
Let $n$ be a positive integer. We provide an explicit geometrically motivated $1$-Lipschitz map from the space of persistence diagrams on $n$ points (equipped with the Bottleneck distance) into the Hilbert space $\ell^2$. Such maps are a crucial step in topological data analysis, allowing the use of statistical methods (and thus data analysis) on collections of persistence diagrams. The main advantage of our maps as compared to most of the other such vectorizations is that they are coarse and uniform embeddings with explicit distortion functions. This allows us to control the amount of geometric information lost through their application. Furthermore, we also provide an explicit $1$-Lipschitz map from the space of persistence diagrams on $n$ points on a bounded domain into a Euclidean space with an explicit distortion function. We conclude with a differently flavored embedding of the space of persistence diagrams on $n$ points on a bounded domain into $\RR^{n(n+1)}$.

The  maps we construct are fairly simple, with each component depending only on the bottleneck distance to the corresponding ``landmark" persistence diagram. Due to geometric motivation from classical dimension theory, our methods are best described as quantitative dimension theory.  
\end{abstract}

\maketitle

\input{Intro.tex}

\input{Prelim.tex}

\input{Covers.tex}
\input{Scales.tex}

\input{EmbeddingBD.tex}

\input{Conclusion.tex}

\input{Bib.tex}

\end{document}

%% file: Intro.tex
\section{Introduction}

Persistent homology \cite{Edels} is one of the cornerstones of modern applied topology and an increasingly popular technique of topological data analysis (TDA). It is a construction where the well established theoretical background (such as stability theorems) and practical applications in the context of data science exhibit significant interplay to a mutual benefit. One of the major detractions to an even more widespread application of TDA is the fact that the persistence diagrams\textemdash the natural outputs of persistent homology computations\textemdash are not subsets of Hilbert space. This is problematic since the  statistical techniques and general data analysis typically require the structure of a finite dimensional vector space or a Hilbert space in order to apply the corresponding standard tools \cite{Mil, Turner1}. As a result, there have been multiple approaches (including \cite{Adams, Carlsson, Fasy, Bubenik, Carr, Carr1, Kali, Rein}) to map the various variants of the space of persistence diagrams into a Hilbert space or a finite-dimensional Euclidean space. For a review of some of these approaches and comparative studies see \cite{Perea} and \cite{Nanda}.

These approaches are typically Lipschitz maps from a corresponding space of persistence diagrams to a Hilbert space or Euclidean space, a property which is referred to as ``stability'' by some of the papers. In particular, if $f$ is such a map, then 
$$
d(f(A), f(B)) \leq L \cdot \db (A,B),
$$ 
where $d$ is the standard Euclidean or $\ell_2$ distance, $L>0$, and $\db$ is the bottleneck distance between persistence diagrams $A$ and $B$. This is a highly desirable property and implies that the pairwise distances do not increase beyond factor $L$ with application of such $f$. On the other hand, while many of these approaches construct an injective map $f$, there is a notable absence (except for a certain implicit setting in the $1$-Wasserstein distance \cite{Carr1}) of the quantified \textbf{distortion}, i.e., a map $\psi$ implying 
$$
d(f(A), f(B)) \geq \psi(\db (A,B)).
$$ 
This is quite problematic. It means that we have no quantified control on the amount of information lost through $f$. In particular, none of the above-mentioned approaches explains how to distinguish the images of persistence diagrams at the Bottleneck distance above at least a chosen threshold. In other words, there is no control on the discriminative properties of the mentioned approaches. In some of the settings it has been shown that certain distortion functions may not exist \cite{Bauer, Bell, Bub, MitV, Wag}, see also \cite{Weighill}. Such results typically follow from non-embedability results. 

In this paper we build on our previous work \cite{MitV, MitVCor} on the existence of coarse embeddings of the space of persistence diagrams $\D{n}$ on $n$ points. The mentioned works in particular imply that a distortion map exists above certain thresholds. In this paper we extend this existential result to all scales and turn it into explicit maps and provide explicit distortion functions.
The \textbf{main results} of this paper are explicit maps from $\D{n}$ (equipped with the Bottleneck metric) into Hilbert space:
\begin{itemize}
 \item Theorem \ref{ThmGluingScales}: a general framework for $1$-Lipschitz uniform and coarse embedding into Hilbert space. The illustrative example   \ref{ExUnif}  provides explicit distortion functions.
 \item Theorem \ref{ThmBoundedDis}: a general framework for maps of $\Dn$ on bounded domains into Euclidean space with controlled distortion. An explicit case is presented in Example \ref{ExampleFinDim}. 
\end{itemize}
Our maps are fairly simple and easily implementable. They consist of various combinations of bottleneck distances to specific persistence diagrams\textemdash which we think of as \textbf{landmark diagrams}. As a secondary result we prove that our framework may not yield an injective map into any Euclidean space (Proposition \ref{PropNotInj}). We then construct an embedding of $\Dn$ restricted to a bounded domain into Euclidean space (Example \ref{ExInjectMap}) using a different construction.

\textbf{Related results.} The following are some of the (non-)embedability results for spaces of persistence diagrams in the bottleneck metric:
\begin{enumerate}
 \item The space of persistence diagrams on finitely many points does not isometrically embed into Hilbert space \cite{Turner1}.
  \item The space of persistence diagrams on finitely many points does not coarsely embed into Hilbert space \cite{Bub, MitV}.
  \item The space of persistence diagrams on at most $n$ points does not bi-Lipschitzly embed into any finite-dimensional Euclidean space \cite{Bauer}.
  \item The space of persistence diagrams on at most $n$ points is of asymptotic dimension $2n$ and thus coarsely embeds into Hilbert space \cite{MitV, MitVCor}.
    \item The space of persistence diagrams on at most $n$ points  bi-Lipschitzly embeds into Hilbert space \cite{Bate}.
\end{enumerate}
Only the last two results are positive and neither of their proofs is by explicit geometric construction. We note that while \cite{Bate} considers the Wasserstein metric, the Bottleneck and Wasserstein metrics are coarsely equivalent when one restricts to diagrams of up to $n$ points (see  \cite{MitV}, Proposition 3.1). On a similar note, explicit distortions for a kernel of persistence diagrams on at most $n$ points on a bounded domain were obtained for the $1$-Wasserstein norm on persistence diagrams in \cite{Carr1} (see also the subsequent \cite{Divol}). Using the induced uniform embedding and the bi-Lipschitz equivalence of the $1$-Wasserstein and Bottleneck distances on the space of persistence diagrams on $n$ points, one could deduce a uniform embedding of the space of persistence diagrams on at most $n$ points on a bounded domain.  

On a related note, in \cite{Zava} the author shows that the space of isometry classes of metric spaces with at most $n$ points endowed with the Gromov-Hausdorff distance is of asymptotic dimension $n(n-1)/2$ and thus coarsely embeds into Hilbert space.

In this paper we provide an explicit geometric construction of the embedding of (4) in the coarse and uniform settings, along with their quantified approximations mapping into Euclidean space.

\textbf{On the origin of our construction.} In our previous work \cite{MitV, MitVCor} we showed that $\D{n}$ coarsely embeds into Hilbert space by proving it is of asymptotic dimension $2n$. This amounted to proving the existence of certain covers of $\D{n}$. In classical dimension theory, dimension is encoded by existence of certain open covers and one can assign appropriate partitions of unity subordinated to such covers and combine them to obtain useful maps to simplicial complexes. A similar approach is also used in the standard proof of the Nerve Theorem. However, while the mentioned embeddings have in the past been proved to admit appropriate upper and lower bounds on distortions,  specific quantifications have never been carried out\textemdash to the best of our knowledge. The main results of this paper are obtained by choosing explicit covers and functions imitating partitions of unity, for which we can define explicit maps and prove explicit distortions. To the best of our knowledge, this is the first time that these parameters have been optimized for computational purposes. Our construction can thus be referred to as one of the first examples of quantitative dimension theory and, in particular, the first incorporation of the dimension theory construction into the topological data analysis.

The structure of the paper is the following:
\begin{itemize}
\item Section 2: Preliminaries.
\item Section 3: Treatment of maps arising from covers in general and in our specific case. 
\item Sections 4: Theorem \ref{ThmGluingScales} (coarse and uniform embeddings into the Hilbert space) along with the corresponding  illustrative Examples  \ref{ExUnif}.
\item Sections 5: Theorem \ref{ThmBoundedDis} (maps of diagrams on bounded domains into Euclidean spaces) and the corresponding Example \ref{ExampleFinDim}. 
\item Section 6: Discussion on our specific maps, choices, and potential improvements, along with the final description of maps.
\end{itemize}

%% file: Prelim.tex

\section{Preliminaries}
\label{SectP}

\subsection{Persistence diagrams}

In this subsection we recall the structure of the space of persistence diagrams. For a classical introduction to these objects see \cite{Edels} (see also Remark \ref{RemDB}). For most of our purposes we will follow the approach below introduced in \cite{MitV}, \cite{MitVCor}.

Notation $\D{1}$ represents the \textbf{space of persistence diagrams} on $1$ point. As a set it equals $\D{1}=T \cup \{\Delta\} $ where $ T=\{(x_1,x_2)\in \RR^2 \mid x_2 > x_1 \geq 0\} $, while $\Delta$ is an additional point representing the diagonal $\{(x,x)\in \RR^2 \mid x \geq 0\}$. On the space $\D{1}$ we define the \textbf{Bottleneck distance} $\db$ as follows:
\begin{itemize}
 \item  $\db((x_1, x_2), \Delta) = |x_1-x_2|/2$, and
 \item the distances on $\D{1}$ is
 $$
 \db((x_1, x_2), (y_1, y_2))= \min \Big\{  
\max \big\{ |x_1-y_1|, |x_2-y_2|\big\},
 \max \big\{
 \db((x_1, x_2),\Delta), 
 \db((y_1, y_2),\Delta)
 \big\}\Big\}.
 $$
\end{itemize}
The definition of $\db$ in the last line is the minimum of two terms. The first term is the $d_\infty$ distance between the points and corresponds to the matching between the points in the standard definition of the bottleneck metric. The second term corresponds to the matching of both points to the diagonal $\Delta$.

Fix $n\in \{1,2,\ldots\}$. We define the max metric $\db^n$ on the product  space $(\D{1}^n, \db^n)$ in the usual way, by
$$
\db^n((\bar{x}_1, \bar{x}_2, \ldots, \bar{x}_n),(\bar{y}_1, \bar{y}_2, \ldots, \bar{y}_n))=\max \{\db(\bar{x}_i, \bar{y}_i)\mid i\in \{1,2,\ldots, n\}\}.
$$

The symmetric group on $n$ elements, $S_n$, acts on $(\D{1}^n, \db^n)$ by permutations of coordinates. The \textbf{space of persistence diagrams} on $n$ points, $(\D{n}, \db)$, is the quotient  $(\D{1}^n, \db^n)/S_n$:
\begin{enumerate}
 \item Elements of $\D{n}$ are orbits of the $S_n$ action on $\D{1}^n$. In particular, elements of $\D{n}$ are $n$-tuples $[\bar{x}_1, \bar{x}_2, \ldots, \bar{x}_n]$ with identifications $[\bar{x}_1, \bar{x}_2, \ldots, \bar{x}_n]=[\bar{y}_1, \bar{y}_2, \ldots, \bar{y}_n]$ iff $\exists \psi \in S_n: \bar{x}_i = \bar{y}_{\psi(i)}, \forall i$. We will often think of persistence diagrams as multisets, i.e., collections of $n$ points from $\D{1}$ with potential repetitions.
 \item The metric $\db$ is defined as follows (see \cite{Kas} for the same metric):
$$
\db([\bar{x}_1, \bar{x}_2, \ldots, \bar{x}_n],[\bar{y}_1, \bar{y}_2, \ldots, \bar{y}_n])=\min_{\psi \in S_n} \{\db(\bar{x}_i, \bar{y}_{\psi(i)})\mid i\in \{1,2,\ldots, n\}\}.
$$
\end{enumerate}
 
Space $\D{n}$ naturally contains $\D{k}$ for all $k \leq n$: given a diagram on $k$ points, its representative in $\D{n}$ is obtained be adding $n-k$ instances of $\Delta$ to the diagram. This induces isometric embeddings $\D{k} \hookrightarrow \D{n}$.

\begin{Remark}
\label{RemDB}
We also point out a more standard way of defining the Bottleneck distance (this approach is used in Section \ref{PathsDn}). Take two persistence diagrams $A=\{a_1, a_2, \ldots, a_n\}$ and $B=\{b_1, b_2, \ldots, b_k\}$, given by the multisets of points in $\D{1}$. A \emph{partial matching} between $A$ and $B$ is a bijection $\phi\colon A' \to B'$, where $A' \subseteq A$ and $B' \subseteq B$. Let $M(A, B)$ denote the collection of all partial matchings between $A$ and $B$. Then the bottleneck distance can be expressed as
\[
\db(A, B) = \min_{(\phi\colon A' \to B') \in M(A, B)} \max \Big\{    
\max_{a \in A'}\{d_1(a, \phi(a))\}, 
\max_{a \in A \setminus A'}\{d_1(a, \Delta)\}, 
\max_{b \in B \setminus B'}\{d_1(b, \Delta)\}
\Big\},
\]
with $d_1((x_1, x_2),\Delta)=|x_1-x_2|/2$ being the $d_1$ distance to the diagonal $\Delta$. 
\end{Remark}

\subsection{Paths in $\D{n}$}\label{PathsDn}

It turns out that $\cup_n \D{n}$ is \textbf{geodesic} (\cite[Proposition 2.3]{Turner1}, see also \cite{Chow}), i.e., for each pair of diagrams $A,B\in \cup_n \D{n}$ there exists an isometric embedding $\gamma \colon [0,\db(A,B)] \to \cup_n \D{n}$ mapping $0 \mapsto A$ and $\db(A,B) \mapsto B$. The said geodesic is defined as follows (in the notation of Remark \ref{RemDB}). Take a partial matching $(\phi\colon A' \to B') \in M(A, B)$ realizing $\db(A, B)$. Now transform $A$ into $B$ by linearly sending:
\begin{itemize}
 \item each $a\in A'$ towards $\phi(A)\in B'$;
 \item each $a\in A \setminus A'$ towards $\Delta$,
  \item each $b\in B \setminus B'$ towards $\Delta$.
\end{itemize}
Unfortunately though, if $A, B \in \D{n}$, then the geodesic between them may not lie in $\D{n}$ but rather in $\D{2n}$ (this happens when the optimal matching involves the diagonal). Hence $\D{n}$ is not geodesic, as the example with diagrams $\{(1,5)\}, \{(101, 105)\} \in \D{1}$ demonstrates, see Figure \ref{Fig0}. The next lemma shows that each $\D{n}$ is geodesic up to factor $2$.
\begin{Lemma}
 \label{LemmaGeodesic}
Fix $n \in \{1, 2, \ldots\}$.  For each pair of diagrams $A, B$ there is a path $\gamma$ in $\D{n}$ from $A$ to $B$ of length at most $2 \db(A, B)$.
 \end{Lemma}
\begin{proof}
 Imitating the construction of a geodesic preceding the lemma and the notation of Remark \ref{RemDB}, we construct path $\gamma$ from two sections as follows. Take a partial matching $(\phi\colon A' \to B') \in M(A, B)$ realizing $\db(A, B)$. First take the geodesic from $A$ to $B'$, by sending each $a\in A'$ towards $\phi(A)\in B'$,
and each $a\in A \setminus A'$ towards $\Delta$. Second, take the geodesic from $B'$ to $B$ by sending for each $b\in B \setminus B'$, $\Delta$ towards $b$. Each of these two segments is a geodesic in $\D{n}$, hence the concatenation is a path in $\D{n}$ from $A$ to $B$ of length at most $2 \db(A, B)$.
\end{proof}

\begin{figure}
\includegraphics[width=6cm]{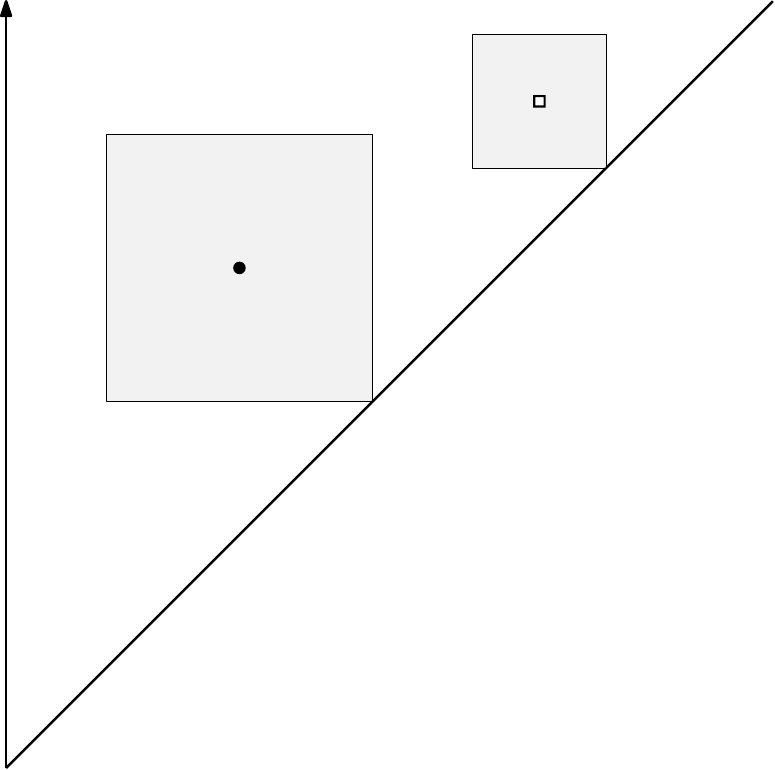}
\includegraphics[width=6cm]{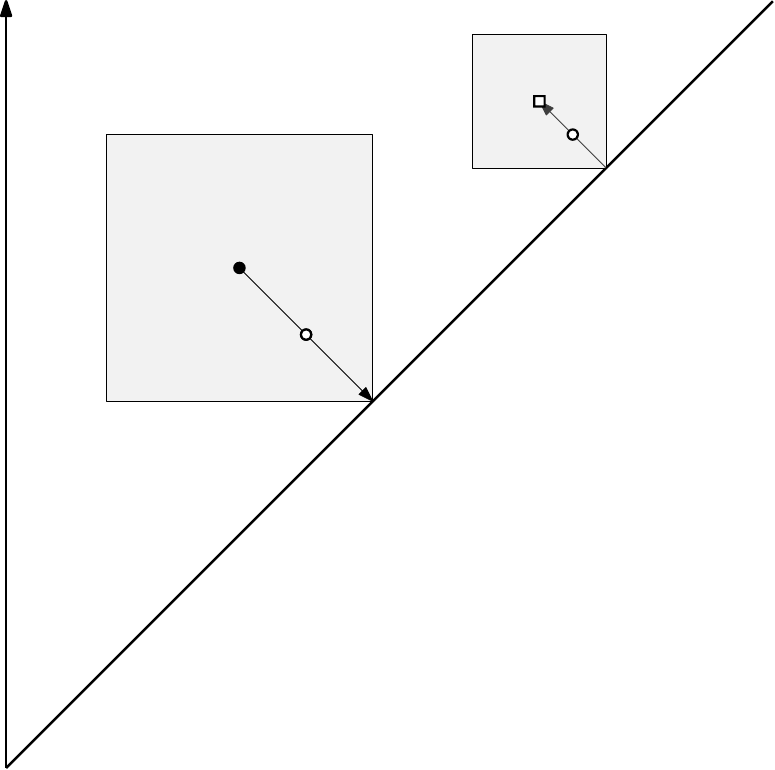}
\caption{On the left are two persistence diagrams in $\D{1}$, one given by $\bullet$, the other with $\square$. The grey squares indicate the closest points of $\Delta$. The geodesic between the two diagrams consists of two-point diagrams, one of which is indicated by the two $\circ$ points. Lemma \ref{LemmaGeodesic} states that if we slide the point of the first diagram to $\Delta$ first, and then slide $\Delta$ towards the point of the second diagram, we obtain a path in $\D{1}$ of length twice the bottleneck distance between them. }
\label{Fig0}
\end{figure}

\subsection{Metric geometry}

In this subsection we recall the concepts of metric geometry that we use throughout the paper.  The following definition summarizes the various concepts of embeddings that we use.

\begin{Definition}
Let $f:X \to Y$ be a function between metric spaces.
\begin{enumerate}
\item $f$ is said to be \textbf{Lipschitz} if there is $\Lambda > 0$ such that $d_Y(f(x_1),f(x_2)) \le \Lambda \cdot d_X(x_1,x_2)$. We occasionally call such map $\Lambda$-Lipschitz.
\item $f$ is  said to be a \textbf{coarse embedding} if  there are non-decreasing functions $\rho_{-}, \rho_+ \colon [0,\infty) \to [0, \infty)$  with $\rho_{-}(d_X(x_1,x_2)) \le d_Y(f(x_1),f(x_2)) \le \rho_{+}(d_X(x_1,x_2)) $ and with $\lim_{t \to \infty} \rho_{-}(t) = \infty$.
\item A coarse embedding $f$ is said to be a \textbf{quasi-isometric} embedding if the functions  $\rho_{-,+}$ are linear.
\item A quasi-isometric embedding is said to be a \textbf{bi-Lipschitz} embedding if the functions  $\rho_{-,+}$  are dilations, i.e., $\rho_{-} (t)=a \cdot t$ and  $\rho_{+} (t)=b \cdot t$, where $a,b >0$.
\item $f$ is   said to be a \textbf{uniform embedding} if  there are non-decreasing functions $\rho_{-}, \rho_+ \colon [0,\infty) \to [0, \infty)$  with $\rho_{-}(d_X(x_1,x_2)) \le d_Y(f(x_1),f(x_2)) \le \rho_{+}(d_X(x_1,x_2)) $, with $\lim_{t \to 0^{+}} \rho_{+}(t) = 0$ and with $ \rho_{-}(t)  > 0 $ whenever $t > 0$ .
\end{enumerate}
\end{Definition}

Coarse embeddings and quasi-isometric embeddings are asymptotic concepts.  They control the change of distances between points at large scales. On the other hand, bi-Lipschitz and uniform embeddings control the change of distances at all scales.

For our applications, we are interested in embeddings in the Hilbert space $\ell^2$, the space of square-summable sequences of real numbers with point-wise operations that make it into a vector space.  



%% file: Covers.tex

\section{Covers and maps at a single scale}
\label{SCover}

Let us fix a scale $R>0$. In this section we:
\begin{itemize}
 \item construct a cover of $\Dn$ at scale $R$ (Definitions \ref{DefCover1} and \ref{DefCoverN});
 \item construct a corresponding map to the Hilbert space $\ell^2$ at scale $R$ (Definitions \ref{DefF1} and \ref{DefFR}).
\end{itemize}

\subsection{On $\D{1}$}
We begin our construction by treating the space of diagrams on one point first. 

\begin{Definition}
 Let $R\GG$ denote the collection of persistence diagrams in $\D{1}$, whose single point is any of the following:
 $$
 \{(m R, n R) \mid m\in \{1,3,5,\ldots\}, n\in \{4,6,8,\ldots\}, n \geq m + 3\}.
 $$
In particular, for each point of this collection there is a corresponding persistence diagram in $\D{1}$. We also define $R\GG^+=R\GG  \cup \{\Delta\}.$
\end{Definition}

\begin{figure}
\includegraphics[width=7cm]{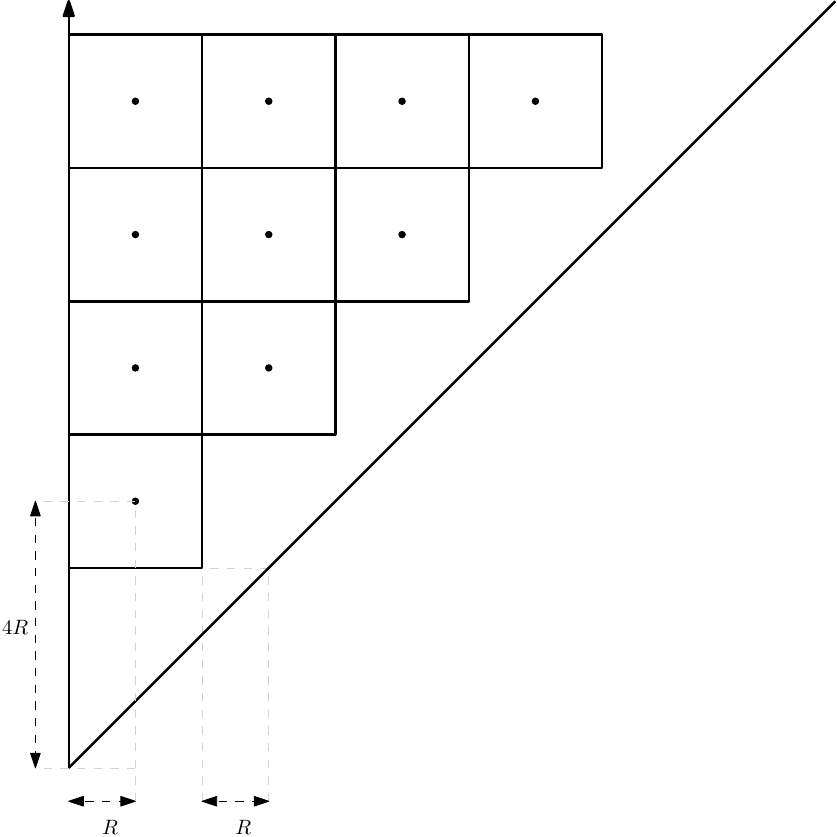}
\caption{$R\GG$.}
\label{Fig1}
\end{figure}

We think of $R\GG$\textemdash a grid at scale $R$\textemdash as a set of landmark diagrams.
The points of $R\GG$  appear as bullets in Figure \ref{Fig1}. We can see that the closed $R$-balls around these points (diagrams) cover the entire space $\D{1}$ except for a zig-zag neighborhood of $\Delta$, which can be covered by the closed $3R/2$-ball around $\Delta$. This observation motivates the following open cover of $\D{1}$.

\begin{Definition}
\label{DefCover1}
Let 
$$
R\UU = 
\{ 
B(p, 3R/2) \mid p\in R\GG^+
\}.
$$
\end{Definition}

Cover $R\UU$ of $\D{1}$ consists of open $3R/2$-balls (squares) around points of $R\GG$ and of the strip $B(\Delta, 3R/2)$  along the diagonal, see Figure \ref{Fig2}.
 
\begin{figure}
\includegraphics[width=4.5cm]{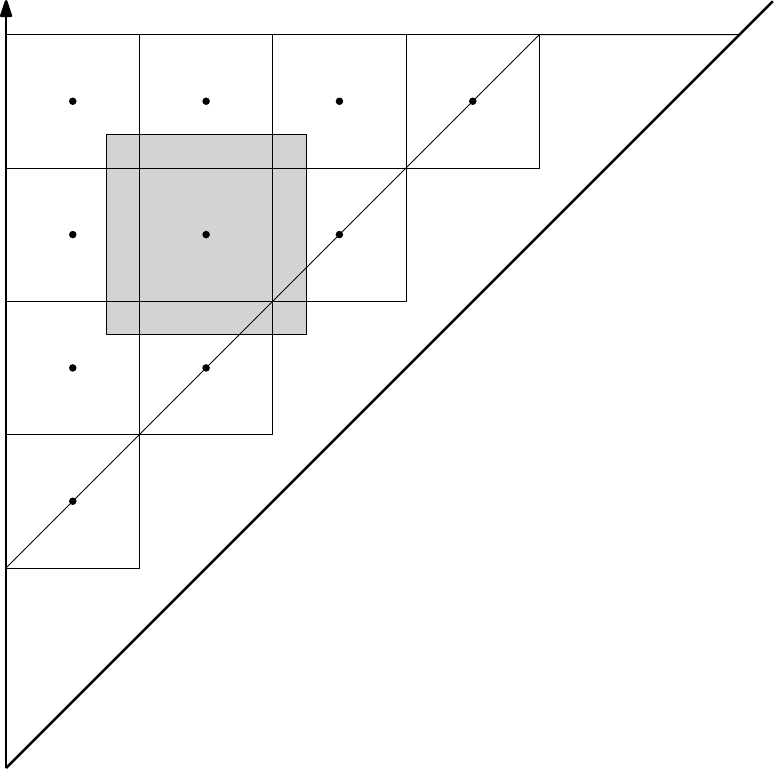}
\includegraphics[width=4.5cm]{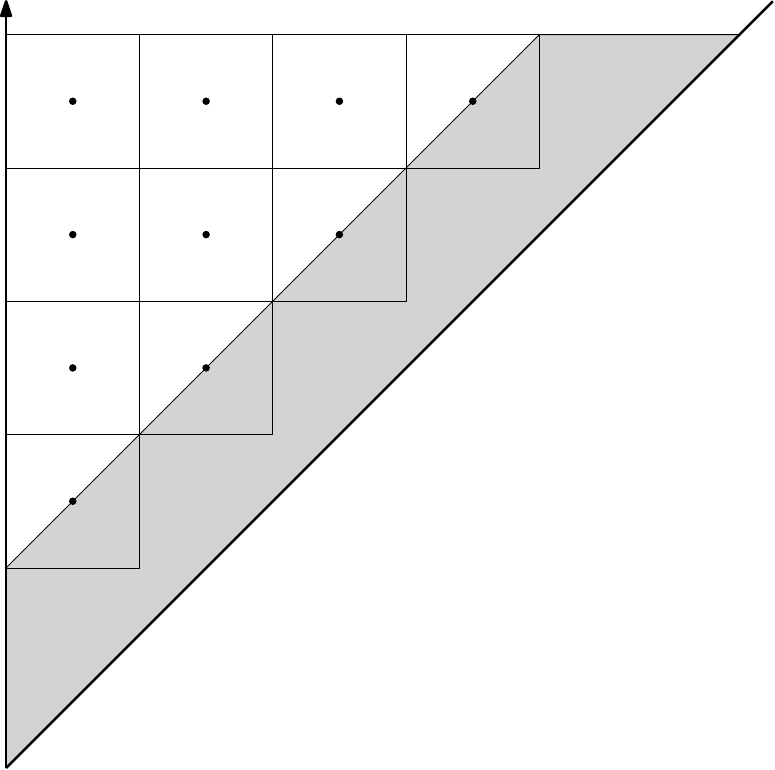}
\includegraphics[width=4.5cm]{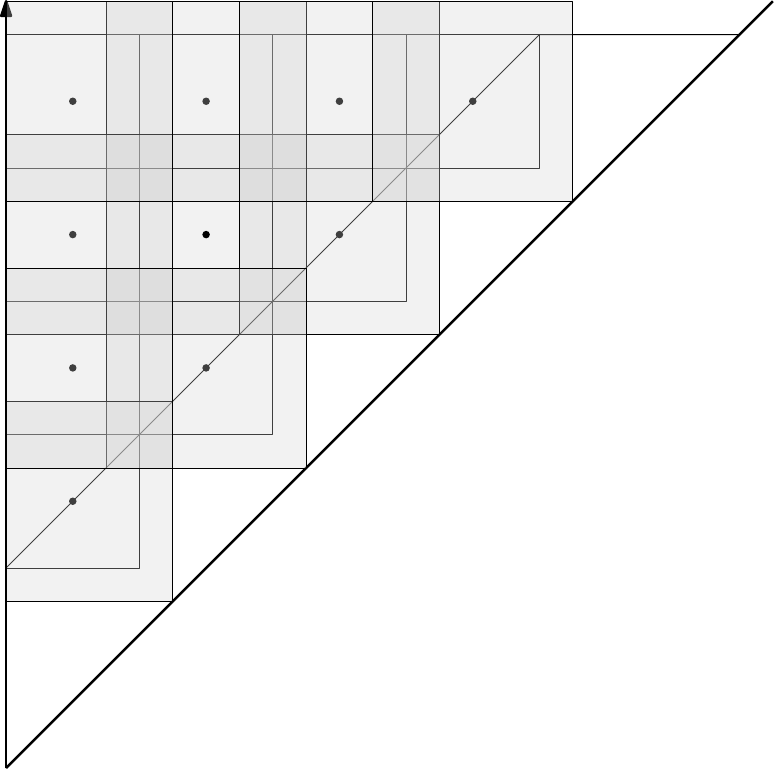}
\caption{Elements of $R\UU$ of Definition \ref{DefCover1}: open $3R/2$-balls around the points of $R\GG$ (left) and the open $3R/2$-ball around the diagonal $\Delta$ (center). A sketch of the cover using opaque squares is given on the right. We can see that the multiplicity of this portion is $4$. The final cover is obtained by adding the ball around $\Delta$. It should be apparent that this addition does not increase the multiplicity.}
\label{Fig2}
\end{figure}

\begin{Lemma}
\label{LemmaMuDiam}
The multiplicity of $R\UU$ equals $4$. The diameter of sets in $R\UU$ is at most $3R$.
\end{Lemma}

\begin{proof}
For the first statement see Figure \ref{Fig2}.
 Diameter of $r$-balls is at most $2 \cdot r$. In our case this gives the bound $3R$. That this is indeed the maximal diameter in our case is demonstrated by any pair of pairwise distant points at bottleneck distance $3R/2$ from the diagonal. 
\end{proof}

We are now in a position to define functions acting as local coordinates on $\D{1}$.

\begin{Definition}
 For each $p\in R\GG^+$ define $\f_{R,p} \colon \D{1}\to [0,R]$ as 
 $$
\f_{R,p}(x)= \max \{3R/2-\db(p,x),0\}.
 $$
\end{Definition}

\begin{Lemma}
\label{LemSum1}
For each $z\in \D{1}$ there exist $p\in R\GG^+$ such that 
$$
\f_{R,p}(z)  \geq R/8.
$$
\end{Lemma}

\begin{proof}
\begin{figure}
\includegraphics{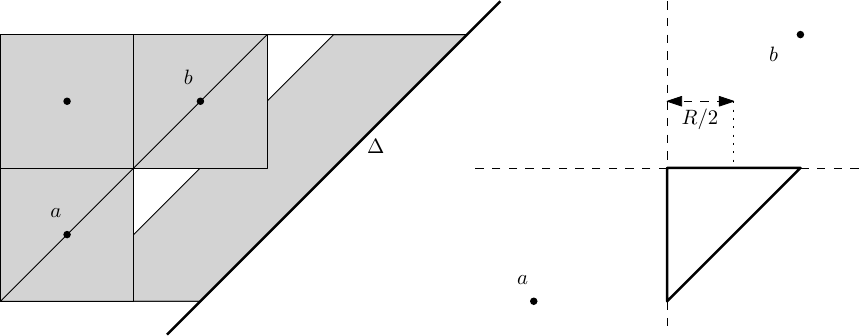}
\caption{Sketch for Lemma \ref{LemSum1}: the region where at least one of the functions $\f_{R,p}$ has value at least $R/2$ (left), and the complementary triangles (right).}
\label{Fig6}
\end{figure}
For each $p\in R\GG^+$ the function $\f_{R,p}$ attains values at least $R/2$ on the closed ball $\overline B(p,R)$. In our decomposition, these balls form the shaded region in Figure \ref{Fig6} (left), whose complement is a collection of triangles close to the diagonal. It remains to provide the proof for points $z$ from all such a triangles. This is demonstrated in Figures  \ref{Fig6a} and \ref{Fig7a} for the lowest such triangle, and on Figure \ref{Fig7} for all other triangles.
\begin{figure}
\includegraphics{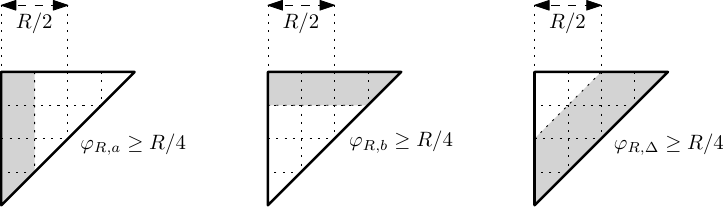}
\caption{Sketch for Lemma \ref{LemSum1}: regions where functions are at least $R/4$ using notation for $a$ and $b$ form Figure \ref{Fig6}.}
\label{Fig7}
\end{figure}
\end{proof}

\begin{figure}
\includegraphics{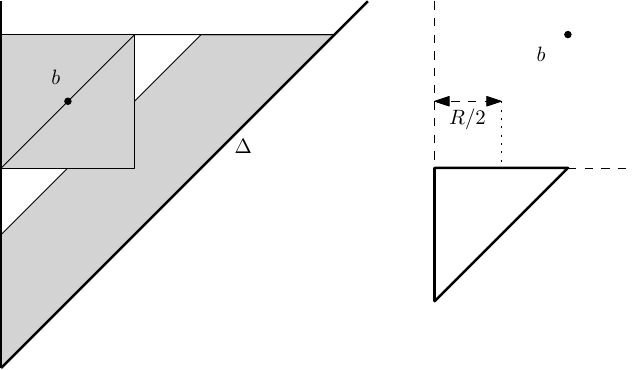}
\caption{Sketch for Lemma \ref{LemSum1}: the bottom part of the region where at least one of the functions $\f_{R,p}$ has value at least $R/2$ (left), and the lowest complementary triangle (right).}
\label{Fig6a}
\end{figure}

\begin{figure}
\includegraphics{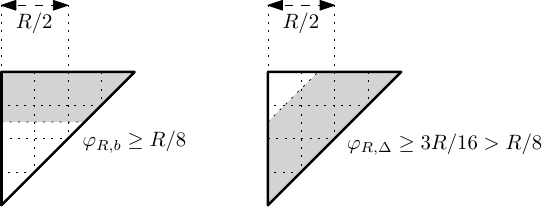}
\caption{Sketch for Lemma \ref{LemSum1}: regions of the lowest complemantary triangle where functions are at least $R/8$ using notation for  $b$ form Figure \ref{Fig6a}. The difference as compared to the argument of Figure \ref{Fig6} is that there is no shaded region coming from the left.}
\label{Fig7a}
\end{figure}

As a sidenote we point out that on the right side of Figure \ref{Fig7a}, the lower bound is in fact $3R/16$. Extending the grey region on this part of the figure and shrinking it in the left part of the image we could thus use the same argument to prove Lemma \ref{LemSum1} for a lower bound slightly higher than $R/8$. However, for the sake of clarity we refrain from such endeavour. 

\begin{Lemma}
\label{LemLip1}
 For each $p\in R\GG^+$ function $\f_{R,p}$ is $1$-Lipschitz.
\end{Lemma}

\begin{proof}
 The lemma is a direct consequence of the triangle inequality.
\end{proof}

We now assemble local coordinates into a single map.

\begin{Definition}
\label{DefF1}
Define $\f_R \colon \D{1}\to [0,3R/2]^{|R\UU|}$ as 
 $$
 \f_R(x)= (\f_{R,p}(x))_{p \in R\GG^+}.
 $$
\end{Definition}

In our setting $|R\UU|$ is (countably) infinite and is indexed by points in $R\GG^+$. The notation $[0,R]^{|R\UU|}$ above indicates, that $\f_R$ is defined by declaring one of its component for each element of $|R\UU|$. As $R\UU$ is of multiplicity $4$, for each point $x\in \D{1}$ at most four coordinates of $\f_R(x)$ are non-trivial, so there is no issue with convergence and the image of $\f_R$ is indeed a subset of the Hilbert space $\ell^2$ --- the space of square summable real sequences with the usual inner product.. In conjunction with Lemma \ref{LemLip1} the last observation also implies the following lemma by the Pythagorean theorem.

\begin{Lemma}
\label{LemLip2}
 For each $R>0$, the function $\f_R$ is $2\sqrt{2}$-Lipschitz and $||\f_R||_2 \geq R/8$.
\end{Lemma}

\begin{proof}
Choose $x,y\in \D{1}$ and a path  $\gamma$ between them of length at most $2\db(x,y)$, as described in Lemma \ref{LemmaGeodesic}. As $R\GG$ is of multiplicity $4$, we can partition $\gamma$ into segments (geodesics) such that for each segment there is an associated collection of four sets of $R\UU$ covering that segment. In particular, 
\begin{enumerate}
 \item the segments are divided by points $x_i, i\in \{i=0,1,\ldots, m\}$ in that order with $x_0=x, x_m=y$, and thus
 $$
2\db(x,y) \geq \sum_{i=0}^{m-1}\db(x_i, x_{i+1});
 $$
 \item for each $i$ there exist $a_1, a_2, a_3,a_4 \in R\GG^+$ such that for each $z\in \{x_i, x_{i+1}\}$ the following holds:
 $$
p\in R\GG^+, p\neq a_j \implies \f_{R,p}(z) =0.
 $$
 \end{enumerate}
Along each segment only the four coordinates of $\f_R$ corresponding to the associate elements of $R\GG$ change. As each of these coordinates is $1$-Lipschitz by Lemma \ref{LemLip1}, the map $\f_R$ is $2\sqrt{2}$-Lipschitz along this segment by the Pythagorean Theorem and thus 
$$
||\f_R(x_i)-\f_R(x_{i+1})||_2 \leq 2\sqrt{2} \db (x_i, x_{i+1}).
$$
This formula, point (1) above, and the triangle inequality imply 
$$
||\f_R(x)-\f_R((y))||_2 \leq \sum_{i=0}^{m-1}||\f_R(x_i)-\f_R(x_{i+1})||_2 
\leq 
2\sqrt{2} \sum_{i=0}^{m-1} \db (x_i, x_{i+1})= 2\sqrt{2}\db(x,y), 
$$
hence $\f_R$ is $2\sqrt{2}$-Lipschitz.

 The bound $||\f_R||_2 \geq R/8$ holds by Lemma \ref{LemSum1}.
\end{proof}


\begin{Lemma}
\label{LemmaDistFun1}
 $\db(x,y)\geq 3R \implies {\lVert \f_R(x)-\f_R(y) \rVert}_2  \geq R\sqrt{2}/{8}$.
\end{Lemma}

\begin{proof}
 By Lemma \ref{LemmaMuDiam} there is no element of $R\UU$ containing both $x$ and $y$. Consequently we have $\f_{R}(x) \perp \f_{R}(y)$ and the conclusion follows from Lemma \ref{LemLip2} by the Pythagorean theorem.
\end{proof}

We summarize the distortion of the map $\f_R$ by Figure \ref{Fig8}. 

\begin{figure}
\includegraphics{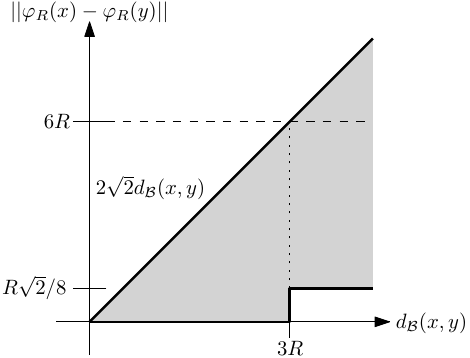}
\caption{A summary of Lemmas \ref{LemLip2} and \ref{LemmaDistFun1}. The grey region represents the potential values of $||\varphi_R(x)-\varphi_R(y) ||_2$  according to the said two lemmas. As the cover is of multiplicity $4$, and $||\varphi_{R,p}||_2 \leq 3R/2$, we actually also get $||\varphi_R(x)-\varphi_R(y) ||_2 \leq 6R$, as indicated by the dashed line. This upper bound for $||\varphi_r||_2$ suggests that in order to obtain a coarse embedding, maps $\varphi_R$ at multiple scales should be combined. However, as this upper  does not improve the Lipschitz constant of the upper bound on distortion, it is not used in the subsequent figures of the distortions nor in the final bounds. Its involvement in the argument would yield a slightly lower, but more complicated upper bound on distortion, hence we kept the current formulation of the Lipschitz upper bound on distortion, as is standard in the community. }
\label{Fig8}
\end{figure}

\subsection{On $\D{n}$}
Fix $n\in \{1,2,\ldots\}$. We now generalize the observations of the last subsection to the space of persistence diagrams on $n$ points.

\begin{Definition}
 Let $R\GG^n$ denote the collection of persistence diagrams in $\D{n}$, whose points are contained in $R\GG$. In a similar manner let $R\GG^{n+}$ denote the collection of persistence diagrams in $\D{n}$, whose points are contained in $R\GG^+$.
\end{Definition}

\begin{Definition}
\label{DefCoverN}
Let 
$$
R\VV = 
\{ 
B(p, 3R/2) \mid p\in R\GG^{n+}
\}.
$$
\end{Definition}

Let us quickly demonstrate that  $\VV$ is a cover of $\D{n}$. Suppose $[x_1, x_2, \ldots, x_n]\in \D{n}$. For each $i$ there exists $g_i\in R\GG^+$ such that $\db(x_i, g_i) \leq R$. Then
$$
\db([x_1, x_2, \ldots, x_n],[g_1, g_2, \ldots, g_n]) \leq \max \db(x_i, g_i)\leq R, 
$$
hence $[x_1, x_2, \ldots, x_n]\in B([g_1, g_2, \ldots, g_n],R).$

\begin{Lemma}
\label{LemmaMuDiamN}
The multiplicity of $R\VV$ equals $4^n$. The diameter of sets in $R\VV$ is at most $3R$.
\end{Lemma}

\begin{proof}
A simple counting argument demonstrates that $[x_1, x_2, \ldots, x_n]\in \D{n}$ is contained in at most $4^n$ sets of $R\VV$ as each $x_i$ is contained in at most $4$ sets of $R\UU$. The argument for diameter is the same as in the proof of Lemma \ref{LemmaMuDiam}.
\end{proof}

\begin{Definition}
 For each $p\in R\GG^{n+}$ define $\f_{R,p} \colon \D{n}\to [0,R]$ as 
 $$
\f_{R,p}(x)= \max \{3R/2-\db(p,x),0\}.
 $$
\end{Definition}

\begin{Lemma}
\label{LemSum1N}
For each $z\in \D{n}$ there exist $p\in R\GG^{n+}$ such that 
$$
\f_{R,p}(z)  \geq R/8.
$$
\end{Lemma}

\begin{proof}
 Choose $[x_1, x_2, \ldots, x_n]\in \D{n}$. For each $i$ there exists $g_i \in R\GG$ such that $\f_{R,g_i}(x_i) \geq R/8$ by Lemma \ref{LemSum1}, or equivalently, $\db(x_i, g_i)\leq 11R/8$. By the definition of the bottleneck distance we have 
 $$
 \db([x_1, x_2, \ldots, x_n],[g_1, g_2, \ldots, g_n]) \leq 11R/8,
 $$
 which implies $\f_{R,[g_1, g_2, \ldots, g_n]}([x_1, x_2, \ldots, x_n]) \geq R/8$.
\end{proof}

\begin{Lemma}
\label{LemLip1N}
 For each $p\in R\GG^{n+}$ function $\f_{R,p}$ is $1$-Lipschitz.
\end{Lemma}

\begin{proof}
 As Lemma \ref{LemLip1}, this lemma is also a direct consequence of the triangle inequality.
\end{proof}

\begin{Definition}
\label{DefFR}
Define $\f_R \colon \D{n}\to [0,3R/2]^{|R\VV|}$ as 
 $$
 \f_R(x)= (\f_{R,p}(x))_{p \in R\GG^{n+}}.
 $$
\end{Definition}

\begin{figure}
\includegraphics[width=4.5cm]{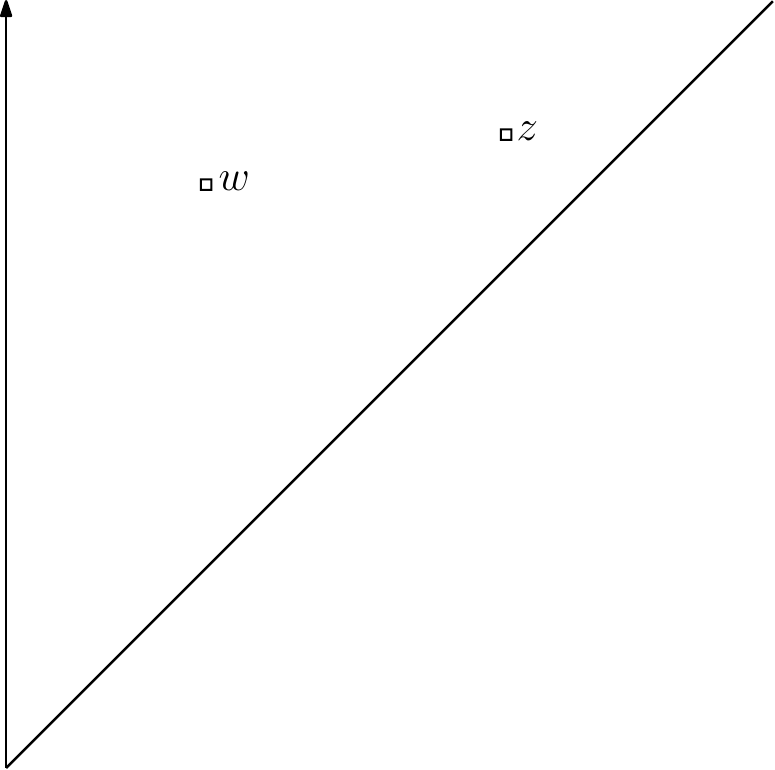}
\includegraphics[width=4.5cm]{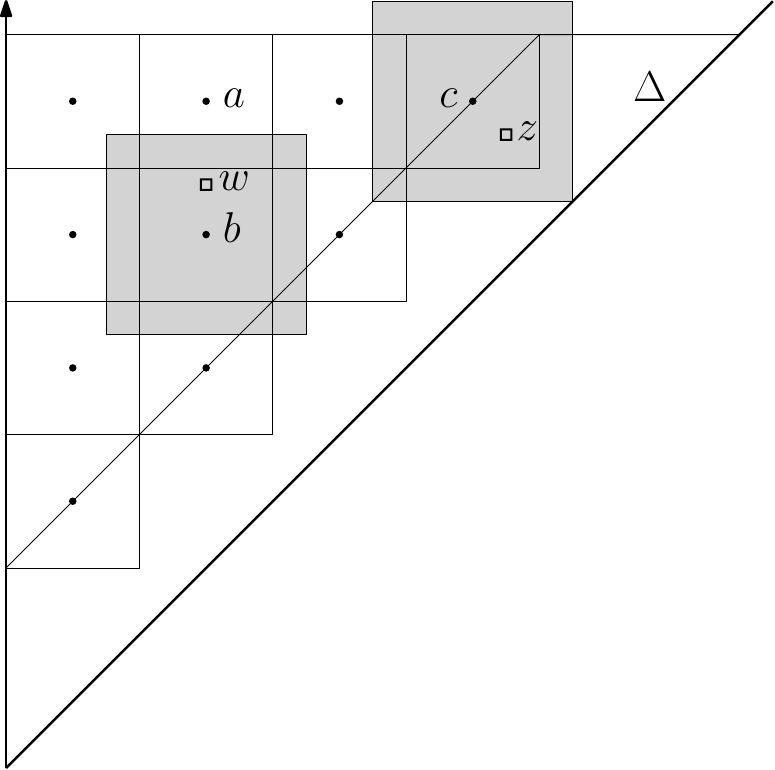}
\includegraphics[width=4.5cm]{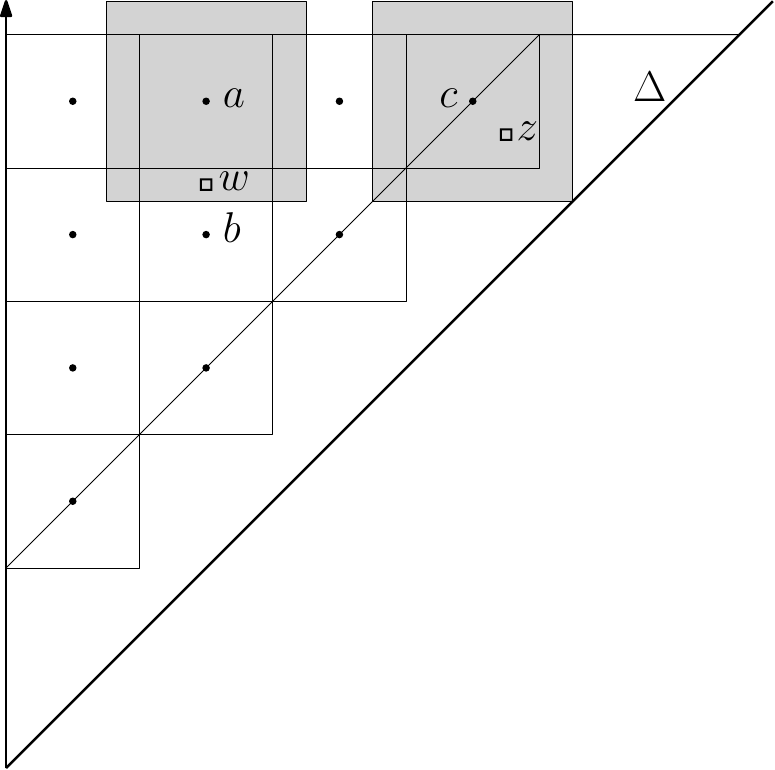}\\
\includegraphics[width=4.5cm]{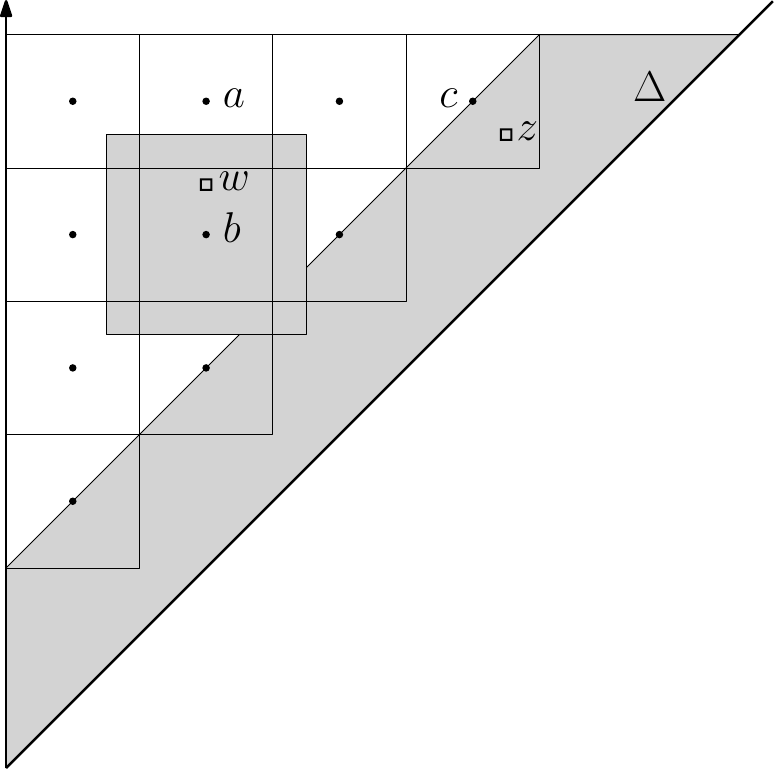}
\includegraphics[width=4.5cm]{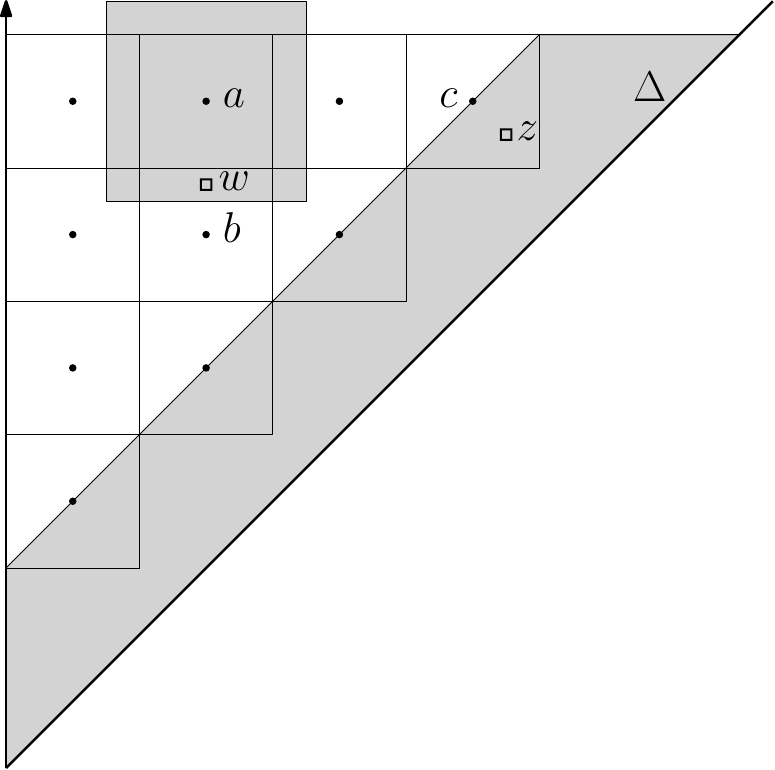}
\caption{A persistence diagram $x$ containing points $w$ and $z$ is given on the upper left portion of the figure. By Definition \ref{DefFR}, $\f_R(x)$ is a vector in Hilbert space with only four non-trivial coordinates, indicated by the other four parts of the figure. These correspond to the following diagrams of $R\GG^+: 
\{b,c\}, \{a,c\}, \{b, \Delta\}, \{a, \Delta\}$, which are listed in the same order as on the figure (see shaded portions).}
\label{Fig10}
\end{figure}

For an example explaining $\f_R$ see Figure \ref{Fig10}.  Again we note that $R\VV$ is indexed by $R\GG^{n+}$. As $R\VV$ is of finite multiplicity $4^n$ its image lies in the Hilbert space. 

\begin{Lemma}
\label{LemLip2N}
 For each $R>0$, the function $\f_R$ is $2^{n+1/2}$-Lipschitz and $||\f_R||_2 \geq R/8$.
\end{Lemma}

\begin{proof}
Choose $x,y\in \D{n}$. As in Lemma \ref{LemLip2} we use Lemma \ref{LemmaGeodesic} to
\begin{enumerate}
 \item choose a path in $\D{n}$ between $x$ and $y$ of length at most $2 \db(x,y)$;
 \item partition the path into segments, such that for each segment there is a dedicated collection of at most $4^n$ elements of $R\GG^{n+}$, whose associated functions $\f_{R,p}$ are the only potentially non-trivial functions along the segment;
 \item use Lemma \ref{LemLip1N} to deduce that along each segment $\f_G$ is $2^{n + 1/2}$-Lipschitz as $\sqrt{4^n \cdot 2}=2^{n + 1/2}$, and thus 
 $$
 ||\f_R(x)-\f_R(y)||_2 \leq  2^{n+1/2}\db(x,y), 
 $$
 by the construction of the geodesic. 
\end{enumerate}
 
 Bound $||\f_R||_2 \geq R/8$ holds by Lemma \ref{LemSum1N}.
\end{proof}

\begin{Lemma}
\label{LemmaDistFunN}
 $\db(x,y)\geq 3R \implies {\lVert \f_R(x)-\f_R(y) \rVert}_2  \geq R\sqrt{2}/{8}$.
\end{Lemma}

\begin{proof}
 The proof is the same as that of Lemma \ref{LemmaDistFun1} using Lemma \ref{LemmaMuDiamN} and \ref{LemLip2N}.
\end{proof}

\begin{figure}
\includegraphics{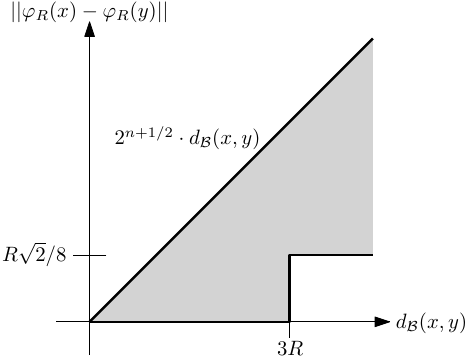}
\caption{A summary of Lemmas \ref{LemLip2N} and \ref{LemmaDistFunN}. The grey region represents the potential values of $||\varphi_R(x)-\varphi_R(y) ||$.}
\label{Fig9}
\end{figure}

%% file: Scales.tex

\section{Incorporation of multiple scales: uniform and coarse embeddings}
\label{SScales}

In this section we assemble the Lipschitz maps constructed at each scale $R>0$ to get a map  to Hilbert space that is both coarse and uniform. We also provide explicit distortions. The general concept of such assembly is provided by Theorem \ref{ThmGluingScales}, while explicit illustrations are given in Example   \ref{ExUnif} (see also Figure \Ref{Fig11}).

\begin{Remark}
\label{Remark}
From Section \ref{SCover} we have, for each scale $R>0$, maps $\f_R$ from $\Dn$ to Hilbert space with the following properties:
\begin{enumerate}
\item  for every $x \in \D{n}$,  ${\lVert \f_R(x) \rVert}_2  \leq  3R/2$. (Definition \ref{DefFR})
\item  for every $x \in \D{n}$,  ${\lVert \f_R(x) \rVert}_2  \ge  R/8$. (Lemma \ref{LemLip2N})
\item for every $x,y \in \D{n}$,  ${\lVert \f_R(x)-\f_R(y) \rVert}_2 \le 2^{n + 1/2} \db (x,y)$. (Lemma \ref{LemLip2N})
\item   $x,y \in \D{n}$ with $\db(x,y)\geq 3R \implies {\lVert \f_R(x)-\f_R(y) \rVert}_2  \geq R\sqrt{2}/{8}$. (Lemma \ref{LemmaDistFunN})
\end{enumerate}
\end{Remark}

Lemma \ref{AssemblyOneLip} explains how to orthogonally assemble a sequence of $1$-Lipschitz maps into a single $1$-Lipschitz map.

\begin{Lemma}
\label{AssemblyOneLip}
Let $(X,d)$ be a metric space with a base point $x_0$, and $f_k \colon X \to \ell_2$ be a sequence of 1-Lipschitz  functions. For a square summable\footnote{$i.e., \displaystyle{\sum_{k=1}^{\infty} w^2_k< \infty}$. } sequence $\{w_k\}_{k=1}^{\infty}$,  the function  
$$
f (x)=\left(\frac{w_k}{ \sqrt{\sum_{i=1}^{\infty} w_i^2}}(f_k (x)-f_k(x_0))\right)_{k=1}^\infty
$$ 
defines a $1$-Lipschitz map to Hilbert space.
\end{Lemma}

\begin{proof}
The image of $f$ is in Hilbert space as 
$$
\lVert f(x) \rVert^2=\sum_{k=1}^{\infty} \frac{w^2_k}{ \sum_{i=1}^{\infty} w_i^2} {\Vert f_k(x)-f_k(x_0) \rVert}^2 \leq {d (x,x_0)}^2.
$$ 
The definition of $f$ makes it $1$-Lipschitz. Specifically,

$$
\lVert f(x)-f(y) \rVert^2 =\sum_{k=1}^{\infty}     \frac{w_k^2}{ \sum_{i=1}^{\infty} w_i^2} \lVert f_{k}(x) - f_{k}(y) \rVert^2 \le  \sum_{k=1}^{\infty}   \frac{w_k^2}{ \sum_{i=1}^{\infty} w_i^2}  {d^2(x,y)} =  d^2(x,y).
$$
\end{proof}

For an interval $J \subset \RR$, map $\chi_{J} \colon \RR \to \RR$ is the indicator map, which equals $1$ on the interval $J$ and $0$ elsewhere.

\begin{Theorem}
[Coarse and uniform embeddings]
 \label{ThmGluingScales}
 Let  $\{R_k\}_{k=1}^\infty$ be an increasing sequence of scales,   $\{\tilde{R}_k\}_{k=1}^\infty$ be  a sequence of scales  decreasing to $0$, and  $ \{w_k\}_{k=1}^\infty$ is a unit vector of weights such that $w_kR_k \to \infty$ . Let $\f_{R_k}$, $\f_{\tilde{R}_k}$ be the maps as described above (Remark \ref{Remark}).
 
 \begin{enumerate}
 
  \item  If   $x_0\in \Dn$ is fixed, then the map $\Phi_1 \colon \Dn  \to \ell_2$ defined as
 $$
 \Phi_1 (x) = \Big(w_k 2^{-n-1/2} (\f_{R_k}(x)-\f_{R_k}(x_0))\Big)_{k=1}^{\infty}
 $$
 is $1$-Lipschitz, and is a coarse embedding with a distortion function 
 $$
 \rho_{-}= \frac{2^{-n-3}}{3}  \sum_{i=1}^{\infty} w_i R_i  \cdot \chi_{[R_i,R_{i+1})}.
 $$

 \item  The map $\Phi_2 \colon \Dn  \to \ell_2$ defined as
 $$
 \Phi_2 (x) = \Big(w_k 2^{-n-1/2}\f_{\tilde{R}_k}(x)\Big)_{k=1}^{\infty}
 $$
 is $1$-Lipschitz, and is a uniform embedding with a distortion function (using $\tilde{R}_0=\infty$)
 $$
 \tilde{\rho}_{-}= \frac{2^{-n-3}}{3} \sum_{i=1}^{\infty}  w_i \tilde{R}_i  \cdot \chi_{[\tilde{R}_i,\tilde{R}_{i-1})}.
 $$

\item The map defined by $\Phi =  \frac{1}{\sqrt{2}}(\Phi_1, \Phi_2)$ is 1-Lipschitz and is a coarse and uniform embedding with a distortion function $\frac{1}{\sqrt{2}} \sqrt{\rho^2_{-} + \tilde\rho^2_{-}}$.
\end{enumerate}
\end{Theorem}

\begin{Remark}
 The map $\Phi_2$ in Theorem \ref{ThmGluingScales} could also have been defined with the  $[-\f_{\tilde{R}_k}(x_0)]$ term, which would facilitate the use of Lemma \ref{AssemblyOneLip}. However, unlike in the case of (1), the use of the said term in (2) is not necessary and we thus omit it for simplicity.
\end{Remark}

\begin{proof}
(1) The map $\Phi_1$ is $1$-Lipschitz by Lemma \ref{AssemblyOneLip}.  
Consider points $x,y \in \D{n}$ with $\db(x,y) \ge R_k$. Then, taking into consideration only scale $k$, 

$$
\lVert \Phi_1(x)-\Phi_1(y) \rVert \ge    2^{-n -1/2} w_k \lVert \f_{\frac{R_k}{3}}(x) - \f_{\frac{R_k}{3}}(y) \rVert \ge  2^{-n-1/2} w_k\frac{1}{3} R_k \sqrt{2}/8 = \frac{2^{-n-3}}{3} w_kR_k.
$$

This gives the required distortion function
$\rho_{-}(t)= \sum_{i=1}^{\infty} \frac{2^{-n-3}}{3} w_i R_i  \cdot \chi_{[R_i,R_{i+1})} (t)$ as a piecewise constant function satisfying
 $\lim_{t \to \infty} \rho_{-}(t) = \infty$.

(2) The fact that $\Phi_2$ is well defined follows from the fact that $\lVert \f_{\tilde{R}_i}(x)\rVert \leq 3 \tilde{R}_i/2, \forall x\in \Dn, \forall i$. The map $\Phi_2$ is $1$-Lipschitz by the same argument as in Lemma \ref{AssemblyOneLip}. The argument from (1) yields the stated distortion function (with the scales now being decreasing, instead of increasing as in (1)), which is piecewise constant and non-zero on positive reals.

Parts (1) and (2) clearly imply (3). 
\end{proof}

\begin{Remark}
 The distortion functions of Theorem \ref{ThmGluingScales} can be easily improved by considering all scales smaller than $R_k$ instead of just the one equal to $R_k$ in the proof. This would yield the distortion functions 
  $$
 \rho_{-}= 
 \sum_{i=1}^{\infty} 
  \sqrt{
  \sum_{j=1}^i
 \left(\frac{2^{-n-3}}{3} w_j R_j \right)^2   
 }\cdot 
 \chi_{[R_i,R_{i+1})}
 $$
 for $\Phi_1$ and
$$
 \rho_{-}= 
 \sum_{i=1}^{\infty} 
  \sqrt{
  \sum_{j=i}^\infty
 \left(\frac{2^{-n-3}}{3} w_j \tilde{R}_j \right)^2   
 }\cdot 
 \chi_{[\tilde{R}_i,\tilde{R}_{i-1})}
 $$
 for $\Phi_2$. In Theorem \ref{ThmGluingScales} and the  illustrative example below we opt for the simpler and more demonstrative (weaker) version. On the other hand, a variant of the stronger version will be used in Section \ref{BEmbedding}.
\end{Remark}


\begin{Example}
\label{ExUnif}
 \textbf{(Explicit illustration of Theorem \ref{ThmGluingScales})}
One can, for instance, take $R_k = k^2$, $\tilde{R}_k = 1/k$ and $w_k = 1/k$. See figure \ref{Fig11}.
\end{Example}

\begin{figure}
\includegraphics[width=7.5cm]{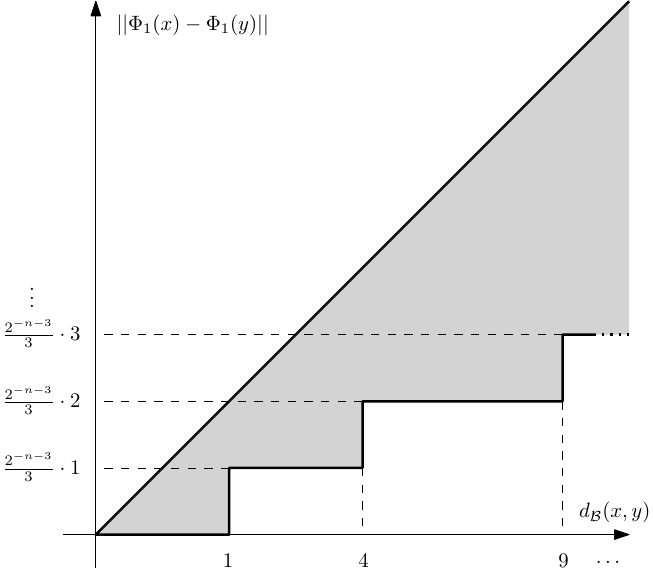}
\includegraphics[width=7.5cm]{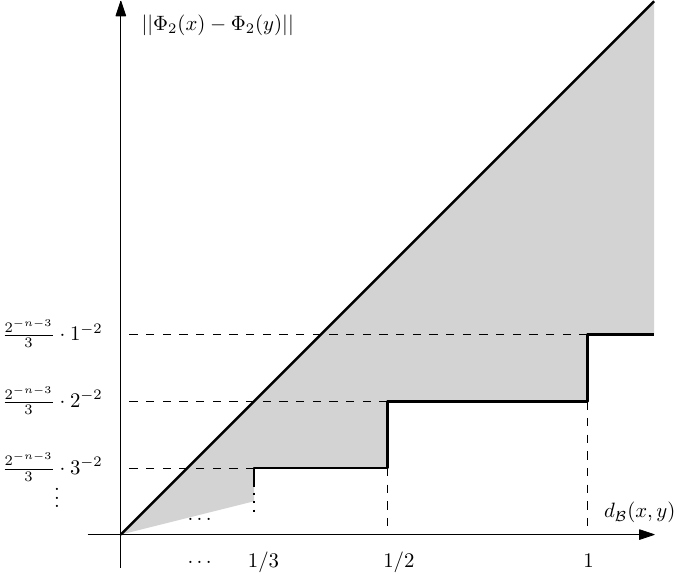}
\caption{A schematic representation of the distortion functions of maps $\Phi_1$   and $\Phi_2$ from Example \ref{ExUnif}. The representations are not to scale. The upper bound is the graph of the identity function as both functions are $\pi^2/6$-Lipschitz. The lower functions are step functions. On the left hand graph, the remaining steps are on the right as the steps are increasing (asymptotic behaviour). On the right hand graph, the remaining steps converge towards zero (local behaviour).}
\label{Fig11}
\end{figure}

%% file: EmbeddingBD.tex

\section{Controlled quasi-embeddings on bounded domains}
\label{BEmbedding}

In this section we discuss how to use our approach in the setting of data analysis. The goal is to map a collection of persistence diagrams into a convenient vector space via a $1$-Lipschitz function $\f$. To that end we will impose a number of restrictions and conditions, relevant to the setting of data analysis. We fix a frame size $L >0$ and restrict our attention to $\Dn \cap [0,L]^2$, which consists of persistence diagrams in $\Dn$ whose coordinates can be represented by pairs in $[0,L]^2$ and the diagonal. We would like to map $\Dn \cap [0,L]^2$ to a Euclidean space. We will thus construct our map by combining only finitely many maps of the form $\f_R$. The last requirement concerns distortion. One of the main advantages of our approach is that we can control the distortion of our map and thus preserve discriminative properties. In the case of finitely many scales this means the following: given a collection of scales $0 < R_1 < R_2 < \ldots < R_N \leq L$, determine the distortion parameters $\dis_i >0 $ so that if diagrams $x,y \in \Dn \cap [0,L]^2$ satisfy $\db(x,y) \geq R_i$, then $||\f(x) - \f(y)|| > \dis_i.$ In the ideal case however, we would like to obtain an explicit lower distortion function $\rho_-$. In this section we will discuss both controls. 

The map $\f$ as discussed above will be obtained as $\f = (w_i 2^{-n }\f_{R_i})_{i-1}^N$, where $(w_1, \ldots, w_N)$ is a unit vector of weights, maps $\f_{R_i}$ were defined in Definition \ref{DefFR}, and $2^{-n-1/2}$ is the normalization factor ensuring the obtained map is $1$-Lipschitz.


\begin{Theorem}
 \label{ThmBoundedDis}
 Let $L>0$ denote the frame size, let $0 < R_1 < R_2 < \ldots < R_N \leq L$ be a chosen sequence of scales, and assume $\bar{w}= (w_1, \ldots, w_N)$ is a unit vector of weights corresponding to the chosen scales. For each $i$ let $\nu_i$ denote the number of elements of the cover $R_i \VV$ intersecting $\Dn \cap [0,L]^2$. Then the map $\f \colon \Dn \cap [0,L]^2 \to \RR^{\nu_1 + \nu_2 + \ldots + \nu_N}$ defined as
 $$
 \f (x) = (w_k 2^{-n -1/2} \f_{R_k}(x))_{k=1}^N
 $$
 is $1$-Lipschitz and satisfies the following:
 if $\db(x,y) \geq R_i$, then  $\lVert \f(x)-\f(y) \rVert    \geq \frac{1}{3 \cdot 2^{n+3}}  \sqrt {  \sum_{k=1}^{i}  {w_k}^2{R_k}^2  }$.

Moreover,  in such a case we have an explicit linear lower bound as below:  
$$
\displaystyle{\rho_-(t)= \left[ \frac{1}{3 \cdot 2^{n+3}} \min \left\{ \min_{2 \leq i \leq N} \left\{   \frac{  \sqrt {  \sum_{k=1}^{i-1}  {w_k}^2{R_k}^2  }  }{R_i-R_1} \right\}, \frac{  \sqrt {  \sum_{k=1}^{N-1}  {w_k}^2{R_k}^2  + w^2_N L^2}  }{L-R_1} \right\}\right] \cdot  (t-R_1)   }
$$ 
when $t > R_1$, and $\rho_-(t)=0$ when $ 0\leq t \leq R_1$.
\end{Theorem}

\begin{proof}
Following the same argument as in  in Lemma \ref{AssemblyOneLip} along with the fact that $\bar{w}$ is an unit vector, the definition of $\f$ makes it $1$-Lipschitz.

If $\db(x,y) \geq R_i$, we have 
$$
\lVert \f(x)-\f(y) \rVert^2 = \sum_{k=1}^{N}    \frac{{w_k}^2}{2^{2n +1}} \lVert \f_{R_k}(x) - \f_{R_k}(y) \rVert^2 \geq \sum_{k=1}^{i}    \frac{{w_k}^2}{2^{2n+1}} \lVert \f_{R_k}(x) - \f_{R_k}(y) \rVert^2 \geq \sum_{k=1}^{i}    \frac{{w_k}^2}{2^{2n+1}} \left(\frac{\sqrt{2}R_k}{3 \cdot 8}\right)^2.
$$

Therefore,  $\db(x,y) \geq R_i$ then $\lVert \f(x)-\f(y) \rVert    \geq \frac{1}{3 \cdot 2^{n+3}}  \sqrt {  \sum_{k=1}^{i}  {w_k}^2{R_k}^2}$.

The explicit linear form of $\rho_{-}(t)$ is obtained by minimizing the slopes of finitely many straight lines formed by joining the point $(R_1,0)$  with the bottom corners of the steps.
\end{proof}

\begin{Example}
\label{ExampleFinDim}
For a chosen number of scales $N> 1$, one of the simplest sets of scales and weights on the interval $[m, M]$ are given by the scales $R_i=m+\frac{M-m}{N}\cdot (i-1)$ and uniform (constant) weights $w_i=1/\sqrt{N}$, where $i=1, \cdots, N$. As earlier, we define a map $\f: \D{n} \to \RR^{N \cdot 4^n}$ by $\f= \left(\frac{1}{2^{n+1/2}\sqrt{N}}\f_{R_i} \right)_{i=1}^N$, which is $1$-Lipschitz. Moreover, this gives   $\rho_{-} (t)= \sum_{i}^{N} \alpha_i  \cdot \chi_{[a_i,b_i]} (t)$, where $a_i=R_i=m+\frac{M-m}{N}\cdot (i-1)$, $b_i=m+\frac{M-m}{N}\cdot i$ and $\alpha_i= \frac{1}{3. 2^{n+3}\sqrt{N}}  \sqrt {  \sum_{k=1}^{i}  {R_k}^2  }$ with $i=1, \ldots, N$, see Figure \ref{FigPdf1}. 
In this case, the $\rho_{-} (t)$ can be replaced by a linear function $\hat\rho_{-} (t)=\frac{1}{3.2^{n+3}  } \cdot \frac{\sqrt{\lambda}}{a \sqrt{N}} \cdot (t-R_1)$ where $a= \frac{M-m}{m  N}$, and $\lambda$ is determined as follows.  

From Theorem \ref{ThmBoundedDis},   $\lambda$ is the the minimum of  the finite sequence of numbers 
$$
\left\{\frac{1+(1+a)^2+\cdots+(1+(n-1)a)^2}{n^2}\right\}_{n=1}^N = \{f(n)\}_{n=1}^N,
$$
 where  $f(x)=\frac{a^2}{3}\cdot x+(1-a+\frac{a^2}{6})\cdot \frac{1}{x} + (a-\frac{a^2}{2})$. If we assume $0 <a \leq 1$ (which holds if $N \geq \frac{M-m}{m})$,  for $x>0$ the function $f$ has positive concavity and has a minimum at $\mu=\sqrt{\frac{3}{a^2} -\frac{3}{a}+\frac{1}{2}}$. Therefore, 

\begin{equation*}
\lambda =\begin{cases}
       \min \{ f(\lfloor \mu \rfloor) ,  f(\lceil \mu \rceil) \} \quad & \text{if    } \, \mu > 1 \\
       f(1) \quad &\text{if  } \, \mu \leq 1 \\
     \end{cases}
\end{equation*}
\end{Example}

\begin{figure}
\begin{center}
\includegraphics[width = 12cm]{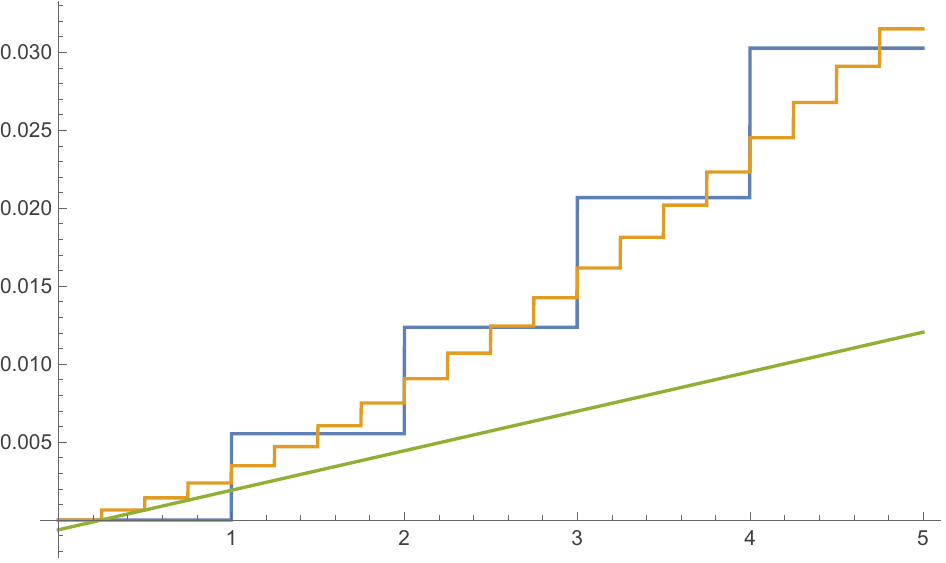}
\caption{An example of lower bounds $\rho_-$ on distortion from Example \ref{ExampleFinDim}. The number of points in persistence diagrams is $n=4$. The graphs correspond to the lower bounds on the interval $[0,5]$ with $4$ and $19$ steps (parameter $N$). We see that the larger number of steps yields a higher upper bound for small distances, while the smaller number of steps yields a higher upper bound for larger distances. The green graph represents the linear distortion $\hat\rho_{-}$.}
\label{FigPdf1}
\end{center}
\end{figure}

\subsection{On injectivity}

In the ideal case we would prefer map $\f$ to be a bi-Lipschitz embedding. However, map $\f$ cannot be bi-Lipschitz by \cite{Bauer}. Furthermore,  $\f$ is not even injective: there are infinitely many points with coordinates of the form $(a,a + \varepsilon)$ for a fixed $\varepsilon >0$, that are close to $\Delta$ and contained in exactly one element of $R_{i}\UU$ (the one containing $\Delta$) for any given $i$. The collection of such points at a fixed scale $R_i$ is given as the non-shaded region on the right part of Figure \ref{Fig2}. It should be apparent that the intersection of such regions for $i=1,\ldots, N$ contains infinitely many points. Taking two such points $(a, a + \varepsilon)$ and $(a', a' + \varepsilon)$ as diagrams in $\D{1} \cap [0,L]^2$, we see that their images by $\f$ coincide. Their non-trivial coordinates correspond to the diagonal terms at each scale. 

The mentioned argument indicates a more general phenomenon: if a map defined on  $\Dn \cap [0,L]^2$ is obtained by combing finitely many distance-to-a-diagram functions $x \mapsto \db(x, A)$, then it can not be injective. This statement formally follows from Proposition \ref{PropNotInj} and also holds if we additionally incorporate maps $\f_{R, p}$ as in our constructions (i.e., compositions of distance-to-a-diagram functions with any functions on $[0,\infty)$), or any maps whose level lines contain almost entire parallels to $\Delta$. This observation motivates the construction of an injective map from $\Dn \cap [0,L]^2$ into a Euclidean space in Example \ref{ExInjectMap}.

\begin{Proposition}
\label{PropNotInj}
Let $L>0$. For any collection of persistence diagrams $p_1, p_2, \ldots, p_k \in \Dn$, there exist uncountable many different persistence diagrams $a_t \in \Dn \cap [0,L]^2$ with  $t\in \RR$, such that for each $k$, $\db(a_t, p_k)$ is independent of $t$.
\end{Proposition}

Rephrasing the conclusion of Proposition \ref{PropNotInj}, persistence diagrams $a_t$ are indistinguishable when analyzed through their $\db$ distances to persistence diagrams  $p_1, p_2, \ldots, p_k \in \Dn$.

\begin{proof}

 Given a persistence diagram $p$, define the matching $MSupp_p$ support as 
 \[
MSupp_p= \bigcup_{x \in p} \{z \in T \mid d_1(z,x) \leq \db(x, \Delta)\}
 \]
see Figure \ref{Fig4a}.
For any point $z \in \D{1}\setminus MSupp_p$ with $\db(z, \Delta) < \db(p, \Delta)$, the one-point persistence diagram $z$ is at distance $\db(p,\Delta)$ from $p$, as for the point $s_i$ of $p$ furthest from $\Delta$ we have $d_1(s_i, z) > \db(s_i, \Delta) > d_1(z, \Delta)$ by the definition of $MSupp_p$. In particular, $\db(z,p)=\db(p, \Delta)$ is independent of $z$.  

If all of the $p_i$ contain a non-diagonal point, then
\[
\left( \D{1} \setminus \bigcup_{i=1}^k MSupp_{p_i} \right) \ \bigcap  \ \left\{z\in \D{1} \mid \db(z, \Delta) < \db(p_i, \Delta), \forall i \right\} \ \bigcap \ [0,L]^2
\]
contains a non-degenerate square $SQ$ arbitrarily close to the diagonal, and so for all uncountable many points $z$ within it, $d(z,p_i)$ is independent of $z$.

If there exists $i_0$ such that $p_{i_0}=\{\Delta\}$, choose a square $SQ$ corresponding to the other indices $i \neq i_0$ and consider the intersection of $SQ$ with any line parallel to the diagonal, for which the intersection is a non-trivial line segment $L$. Note that $ \db(\Delta, z)$ is independent of $z$. For all the uncountable many points $z$ within $L$, $d(z,p_i)$ is thus independent of $z$.
\end{proof}

\begin{figure}
\includegraphics[width = 6cm]{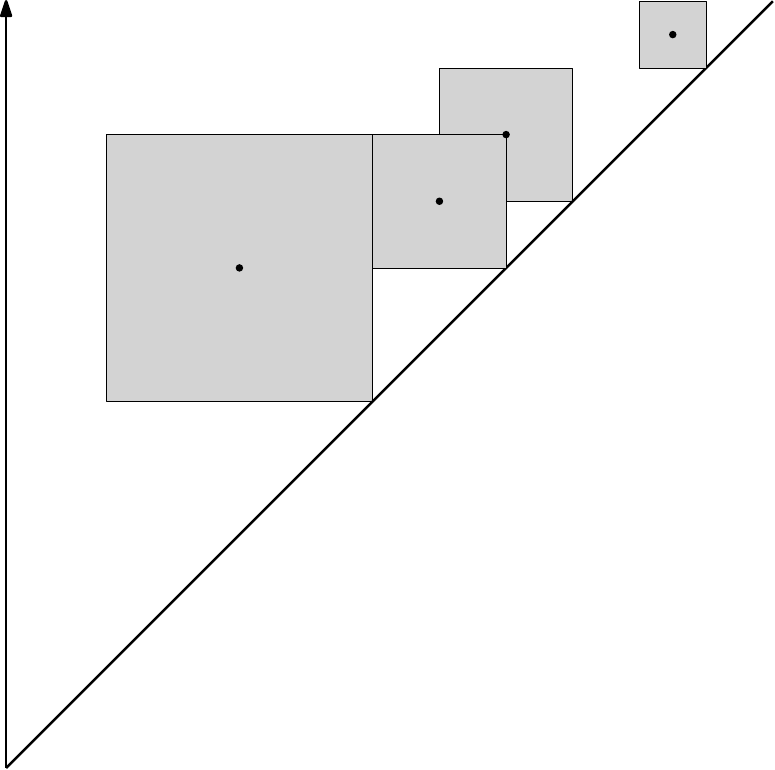} \quad 
\includegraphics[width = 6cm]{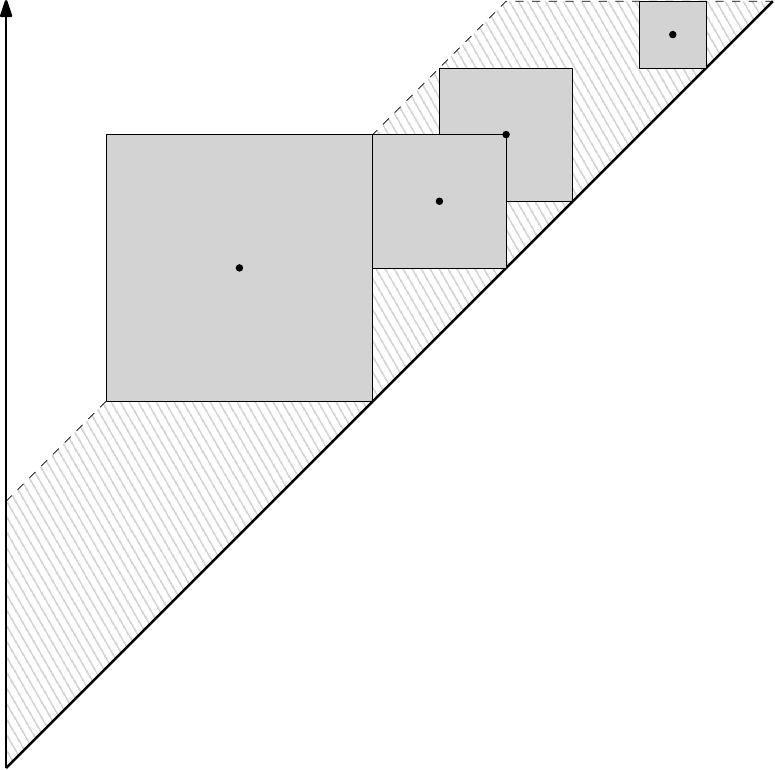}
\caption{A sketch of the proof of Proposition \ref{PropNotInj}. Given a $4$-point persistence diagram $p$ of the bulleted points, the grey are on the left represents $MSupp_p$, while the dashed are on the right represents $\left( \D{1} \setminus \bigcup_{i=1}^k MSupp_{p_i} \right) \bigcap \{z\in \D{1} \mid \db(z, \Delta) < \db(p, \Delta)\}$. Observe that any finite intersection of the latter (dashed) type of sets contains a non-degenerate square of points, which are indistinguishable to the initial finite collection of non-diagonal persistence diagrams.}
\label{Fig4a}
\end{figure}

\begin{Example}
 \label{ExInjectMap}
 In this example we construct a continuous injective map $F \colon \Dn \cap [0,L]^2 \to \RR^{n(n+1)}$. By compactness, $F$ is thus an embedding. 
 
Topologically speaking, $\D{1} \cap [0,L]^2$ is the quotient of $T=\{(x,y)\in [0,L]^2 \mid y \geq x\}$ by the diagonal set $D=\{(x,x)\mid x\in [0, L]\}$. Hence any map defined on $\D{1} \cap [0,L]^2$ arises from a continuous map on $T$ which is constant on $D$. We thus look for functions on $T$ whose level sets are $D$ and lines not parallel with $D$. Given $s < 0$, let $f_s \colon D \to [\pi/4,\pi/2)$ denote the angle with respect to $(s, s)$, i.e., $f_s(x,y) = \textrm{arg}\Big((x,y)-(s, s)\Big)$. The level lines are indicated in Figure \ref{Fig4c}. Observe that $f_s$ is well defined and continuous as $(s,s) \notin T$. 

Define $F_s \colon \Dn \cap [0,L]^2 \to [\pi/4,\pi/2)^n$ by declaring $F_s(p)$ for $p \in \Dn \cap [0,L]^2$ to be the ordered list of the $f_s$-values of the points of $p$, ordered from the smallest to the largest. Note that the same value can repeat in the list. While there is no natural ordering of the points of $p$, the map $F_s$ can be easily seen to be continuous. For example, as $f_s$ is continuous on each single point, the first coordinate of $F_s$ is continuous as it is the minimum of the $n$-many continuous functions $f_s$ applied locally to points of $p$.  

Choose different negative values $s_i, i=1, 2, \ldots, n+1$. We will prove that
$$
F = (F_{s_i})_{i=1}^{n+1} \colon \Dn \cap [0,L]^2 \to \RR^{n \left( n+1\right)}
$$
is injective using a geometric argument. 

Map $F$ is obviously continuous. In order to prove injectivity we take a point $p \in \Dn \cap [0,L]^2$ and show that $p$ can be determined from $F(p)$. For each $i=1, 2, \ldots, n+1$ let $\ell_{i}$ denote the collection of all level lines of $f_{s_i}$ in $T$, whose $f_{s_i}$ value is a coordinate of $F_{s_i}(p)$. Observe that each $\ell_i$ has at most $n$ disjoint lines. If a point of $p$ appears with multiplicity above $1$, then each $\ell_i$ consists of less than $n$ lines. 

We claim that $v \in \D{1} \cap [0,L]^2$ is a point of $p$ (of non-zero multiplicity) iff $v$ is contained in one line from each $\ell_i$. 
\begin{itemize}
 \item Assume $v$ is a point in $\D{1} \cap [0,L]^2$. For each $i$ the $f_{s_i}$, the value of $p$ appears in $F_{s_i}(p)$ and thus the corresponding line of $\ell_i$ contains $v$.
 \item Assume $v$ is contained in one line from each $\ell_i$, i.e., for each $i$ there exists $\l_i \in \ell_i$ containing $v$. By the pigeon-hole principle, there exist $j \neq j'$ such that $\l_j$ and $\l_{j'}$ correspond to the $f_j$ and $f_{j'}$ values of the same point $v'$ from $p$. Thus $v'$ is contained in the intersection $\l_j \cap \l_{j'}$. As this intersection has at most one point we conclude $v = v'$ is a point of $p$.
\end{itemize}
Next we determine the multiplicity $\mu(v)$ of a point $v$ in $p$. To the line $\l_i$ in $\ell_i$ passing through $v$ we assign the multiplicity $\mu(\l_i)$ as the number of occurrences of $f_{s_i}(v)$ in $F_{s_i}$. This is a quantity obtainable from $F(p)$. Observe that $\mu(\l_i)$ is the sum of the multiplicities of the points of $p$ contained in $\lambda$ and thus $\mu(\l_i) \geq \mu(v)$. There exists $j$ such that the only point of $p$ contained in $\l_j$ is $v$, which implies $\mu(v)= \min_{i\in \{1, 2, \ldots, n+1\}} \mu(\l_i)$. 

Wa can thus determine all points of $p$ along with their multiplicity from $F(p)$. Hence $F$ is injective. 

Observe that the same definitions and arguments prove that the definition of $F$ applied to $\Dn$ results in an injective map $\Dn \to \RR^{n(n+1)}$. However, this map is not an embedding: while the one-point diagrams $(n, n+1)$ for $n \in \NN$  form a discrete collection of diagrams in $\Dn$, their $F$ images converge towards $\pi/4=F(\Delta)$.


\begin{figure}
\includegraphics[width = 7cm]{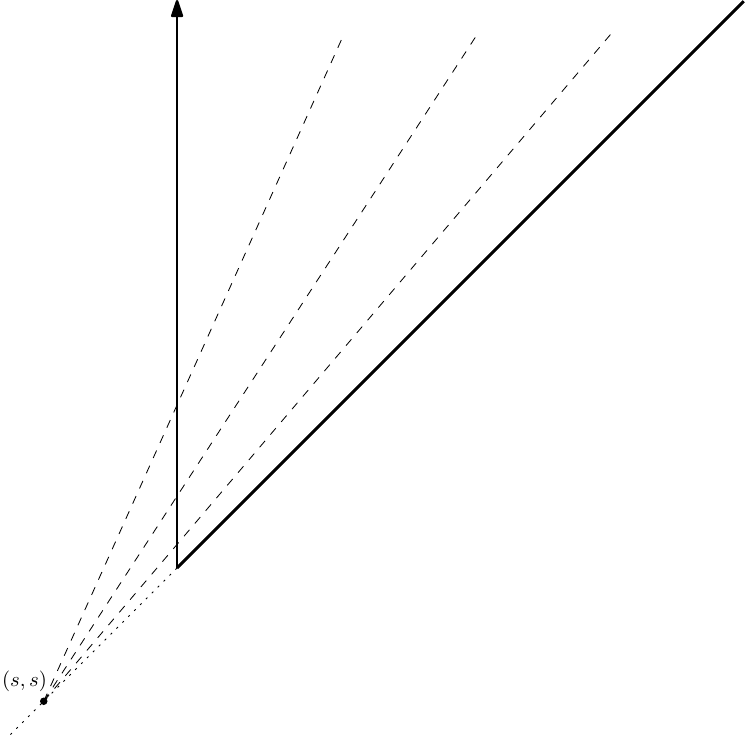}  
\caption{The dashed lines indicate the level lines of $f_s$ from Example \ref{ExInjectMap}. Observe that $f_s(D)=\pi/4$.}
\label{Fig4c}
\end{figure}
 \end{Example}


%% file: Conclusion.tex

\section{Conclusion}

We have provided \textbf{explicit maps} from  spaces of persistence diagrams on $n$ points with the Bottleneck distance \footnote{as mentioned in the introduction, Proposition 3.1 in \cite{MitV} says that similar results are   valid with the Wasserstein distance} into Hilbert and Euclidean space, along with their \textbf{distortions}. Except for Example \ref{ExInjectMap}, the maps are motivated by  constructions in classical dimension theory. In order to obtain explicit maps we had to make a number of choices. Perhaps the most important one is that of a cover $R\mathcal{U}$. It was motivated by a number of factors. 

First, as it consists of balls around certain diagrams it allowed us to use the distance-to-point map in the bottleneck metric, which is a natural continuous function in $\Dn$, and the one we found easiest to implement: constructions in classical dimension theory use the distance-to-the-complement map, which seems complicated in the Bottleneck distance. Furthermore, in dimension theory these maps are usually normalized to obtain a partition of unity (which appear in the proofs of the nerve theorem as well). This appeared as an unnecessary complication which we were able to circumvent using upper and lower bounds on the functions $\f_{R,p}.$ As a result, our maps are of \textbf{very simple form} and ready for computational implementation: \textbf{collections of weighted distances} to specific ``landmark" diagrams. 

The second reason for the cover $R\mathcal{U}$  of $\Dn$ is that it was easy to determine its multiplicity. A drawback is its large multiplicity $4^n$. This contributes to the dominating term $2^{-n}$ in our distortions, making the explicit distortion functions very low. However, our results \cite{MitV, MitVCor} imply that (``at each scale $R$") there exists a cover of $\Dn$ of multiplicity $2n+1$. This would allow us to replace the $2^{-n}$ term with the much preferable $(2n+1)^{-1/2}$ term. Unfortunately, it is not clear how to get a cover of $\Dn$ of multiplicity $2n+1$. More generally, it is not clear how to obtain covers of quotients of spaces while preserving optimal multiplicity. The fact that quotients by finite group actions preserve covering dimension are usually proved using inductive dimension, while the corresponding results for asymptotic dimension use Higson compactification and the corresponding result for the covering dimension \cite{Kas}. In both cases, the proof is not constructive. 

\textbf{Further theoretical advantages.}  The above mentioned adaptations and choices lay out a general framework for explicit discriminative vectorizations in data science. Furthermore, some of the known vectorization techniques, especially persistence images \cite{Adams}, implicitly involve incorporations of vectorizations at multiple scales. Section \ref{SScales} of the current paper now theoretically shows how and why this improves distortion properties over larger scale spans, providing a justification of possible adaptations of this in other vectorization techniques.

\textbf{Practical implementation (Future work)} Our results suggests the obvious course of action: to implement the vectorization and compare it to other known vectorization techniques. In order to improve on the theoretical discrimination obtained in this paper, this will involve numerous practical adaptations, which we are addressing in an ongoing work. The plan is to:
\begin{itemize}
 \item Obtain a more convenient uniform arrangement of the landmark points in the space of persistence diagrams. When done properly, this  is expected to  change the factor $2^{-n}$ in the distortion closer to $(2n-1)^{-1/2}$, as has been explained in a paragraph above.
 \item Optimize the weights used to combine vectorizations at multiple scales.
 \item Our vectorization\textemdash as constructed at present\textemdash is fairly sparse. We are using dimension reduction techniques to bring the output dimension down while controlling the theoretical distoritions.
 \item Provide a thorough statistical analysis and comparison with other vectorization techniques.
\end{itemize}
